\documentclass[11pt,noinfoline]{article}
\RequirePackage[aap,amsthm,amsmath,natbib]{imsart}
\RequirePackage[OT1]{fontenc}
\usepackage{amsfonts,latexsym,amssymb}
\usepackage{graphics,graphicx}
\usepackage{bbm}
\usepackage{enumitem}
\usepackage{subfigure}
\usepackage{epsfig}
\usepackage{xcolor}

\DeclareMathOperator{\tube}{Tube}

\def\R{\mathbb{R}}

\newcommand{\beq}{\begin{eqnarray}}
\newcommand{\eeq}{\end{eqnarray}}
\newcommand{\beqq}{\begin{eqnarray*}}
\newcommand{\eeqq}{\end{eqnarray*}}

\def\cM{\mathcal{M}}

\def\cP{\mathcal{P}}

\def\cX{\mathcal{X}}
\def\cY{\mathcal{Y}}

\def\e{\varepsilon}

\newcommand{\E}{\mathbb{E}} 

\newcommand{\mean}[1] {\E\left\{{#1}\right\}}
\newcommand{\meanx}[1] {\E\{{#1}\}}

\newcommand{\ind}{\boldsymbol{\mathbbm{1}}} 

\newcommand{\iprod}[1]{\left\langle{#1}\right\rangle}
\newcommand{\set}[1]{\left\{#1\right\}}

\newcommand{\norm}[1]{\left\|#1\right\|}
\newcommand{\param}[1]{\left(#1\right)}
\newcommand{\abs}[1] {\left| {#1}\right|}

\newcommand{\prob}[1]{\mathbb{P}\left(#1\right)}
\newcommand{\probx}[1]{\mathbb{P}(#1)}

\newcommand{\tC}{\tilde{C}}

\newcommand{\eps}{\epsilon}

\newcommand{\by}{\mathbf{y}}
\newcommand{\bx}{\mathbf{x}}

\newcommand{\eone}{\mathrm{e}_1}

\newcommand{\cQ}{{\cal{Q}}}

\newtheorem{lem}{Lemma}[section]
\newtheorem{lemma}[lem]{Lemma}

\newtheorem{theorem}[lem]{Theorem}

\newtheorem{corollary}[lem]{Corollary}

\theoremstyle{definition}

\newcommand{\cech}{\v{C}ech }

\newcommand{\iid}{\mathrm{i.i.d.}}

\def\Sk{S_{k,n}}
\def\Lk{L_{k,n}}

\def\bk{\beta_{k,n}}

\def\Skt{\widehat{S}_{k,n}}

\newcommand{\ninf}{n\to\infty}

\newcommand{\limninf}{\lim_{\ninf}}

\def\CC{\check{C}}

\numberwithin{equation}{section}

\def\teq{\triangleq}

\def\Rnc{R_n^{\mathrm{c}}}
\def\cpl{c_{\mathrm{p}}}
\def\cexp{c_{\mathrm{e}}}
\def\cnorm{c_{\mathrm{g}}}
\def\Brn{B_{R_n}}
\def\dpl{\delta_{\mathrm{p}}}
\def\dexp{\delta_{\mathrm{e}}}
\def\dnorm{\delta_{\mathrm{g}}}
\def\fpl{f_{\mathrm{p}}}
\def\fexp{f_{\mathrm{e}}}
\def\fnorm{f_{\mathrm{g}}}
\def\mpl{\mu_{\mathrm{p},k}}
\def\mplo{\mu_{\mathrm{p},0}}
\def\mplh{\hat{\mu}_{\mathrm{p},k}}
\def\mexp{\mu_{\mathrm{e},k}}
\def\mexpo{\mu_{\mathrm{e},0}}
\def\mexph{\hat{\mu}_{\mathrm{e},k}}
\def\Rk{R_{k,n}}
\def\Rke{\Rk^\eps}
\def\Ro{R_{0,n}}
\def\Roe{\Ro^\eps}
\def\So{S_{0,n}}
\def\Sot{\hat{S}_{0,n}}

\def\cxr{\cX_{n,R_n}}

\begin{document}

\begin{frontmatter}
\title{Crackle: The {Persistent}  Homology of Noise}


\begin{aug}
\author{\fnms{Robert J.} \snm{Adler}
\ead[label=e1]{robert@ee.technion.ac.il}
\ead[label=u1,url]{http://webee.technion.ac.il/people/adler}}
\author{\fnms{Omer}
  \snm{Bobrowski}
  \ead[label=e2]{omer@math.duke.edu}
\ead[label=u2,url]{http://www.math.duke.edu/$\sim$omer} }
\and
\author{\fnms{Shmuel} \snm{Weinberger}
\ead[label=e5]{shmuel@math.uchicago.edu}
\ead[label=u5,url]{http://www.math.uchicago.edu/$\sim$shmuel}}


\runauthor{Adler, Bobrowski, Weinberger}
\affiliation{Electrical Engineering, Technion -- Israel Institute of Technology\\
Department of Mathematics, Duke University\\
Department of Mathematics, University of Chicago}
\end{aug}


\begin{abstract}
We study the homology of simplicial complexes built via deterministic rules from a  random set of vertices. In particular,  we show that, depending on the randomness that generates the vertices, the homology of these complexes can either become trivial as the number $n$ of vertices grows, or can contain more and more complex structures.
The different behaviours are consequences of different underlying distributions for the generation of vertices, and we consider three illustrative  examples, when the vertices are sampled from Gaussian, exponential, and power-law distributions in $\R^d$.

We also discuss consequences of our results for  manifold learning with noisy data,  describing  the topological phenomena that arise in this scenario as `crackle', in analogy to audio crackle in temporal signal analysis.

\end{abstract}




\end{frontmatter}



\section{Introduction} \label{sec:crackle:intro}

This paper treats the homology of simplicial complexes built via deterministic rules from a  random set of vertices. In particular, it shows that, depending on the randomness that generates the vertices, the homology of these complexes can either become trivial as the sample size grows, or can contain more and more complex structures.

The motivation for these results comes from applications of topological tools for pattern analysis, object identification, and especially for the analysis of data sets.  Typically, one starts with a collection of points and forms some simplicial complexes associated to these, and then takes their homology.  For example, the $0$-dimensional homology of such complexes can be interpreted as a version of clustering.  The basic philosophy behind this attempt is that topology has an essentially qualitative nature and should therefore be robust with respect to small perturbations.
Some recent references are \cite{aronshtam_vanishing_2010,babson_fundamental_2011,cohen_homotopical_2010,meshulam_homological_2009,pippenger_topological_2006} with two reviews, from different aspects, in \cite{adler_persistent_2010} and \cite{ghrist_barcodes:_2008}. Many of these papers find their {\it raison d'\^etre} in essentially
statistical problems, in which data generates the structures.

An important example occurs in the following manifold learning problem. Let $\cM$ be an unknown manifold embedded in a Euclidean space, and suppose that we are a given a set of independent and identically distributed ($\iid$) random samples $\cX_n = \set{X_1,\ldots,X_n}$ from the manifold. In order to recover the homology of $\cM$, we consider  the homology of
\beq
\label{U:equn}
U \ =\  \bigcup_{k=1}^n B_{\eps}(X_k),
\eeq
where $B_\eps(X)$ is the Euclidean ball, in the ambient space,   of radius $\eps$ about the point $X$. The belief, or hope,  is that, for large enough $n$, the homology of $U$ will be equivalent to that of $\cM$. A confounding issue arises
when the sample points do not necessarily lie on the manifold, but rather are perturbed from it by a random amount.
When this happens, it will follow from our results that the precise distribution behind the randomness plays a qualitatively important role. It is known that if the perturbations come from a bounded or strongly concentrated distribution, then they do not lead to much spurious homology, and the above line of attack, appropriately applied, works. For example,
it was shown in \cite{niyogi_topological_2011}  that for Gaussian noise it is  possible to clean the data and recover the underlying topology of $\cM$ in a way that is essentially independent on the ambient dimension.
Both \cite{niyogi_finding_2008,niyogi_topological_2011} contain results of the form that, given  a nice enough $\cM$, and any $\delta>0$,  there are explicit conditions on $n$ and $\eps$ such that the homology of $U$ is equal to the homology of $\cM$ with a probability of at least $(1-\delta)$.  However, for  other distributions no such results exist, nor, in view of the results of this paper,  are they to be expected.

Figure \ref{fig:mfld} provides an illustrative example of what happens when sampling points from an
annulus and perturbing them with additional noise before reconstructing the annulus as in  \eqref{U:equn}.
In particular, it shows that if the additional noise is in some sense large then  sample points can appear basically anywhere, introducing extraneous homology elements.

\begin{figure}[hc!]
\label{fig:mfld}
\centering
\subfigure[]
{
  \centering
    {\includegraphics[scale=0.2]{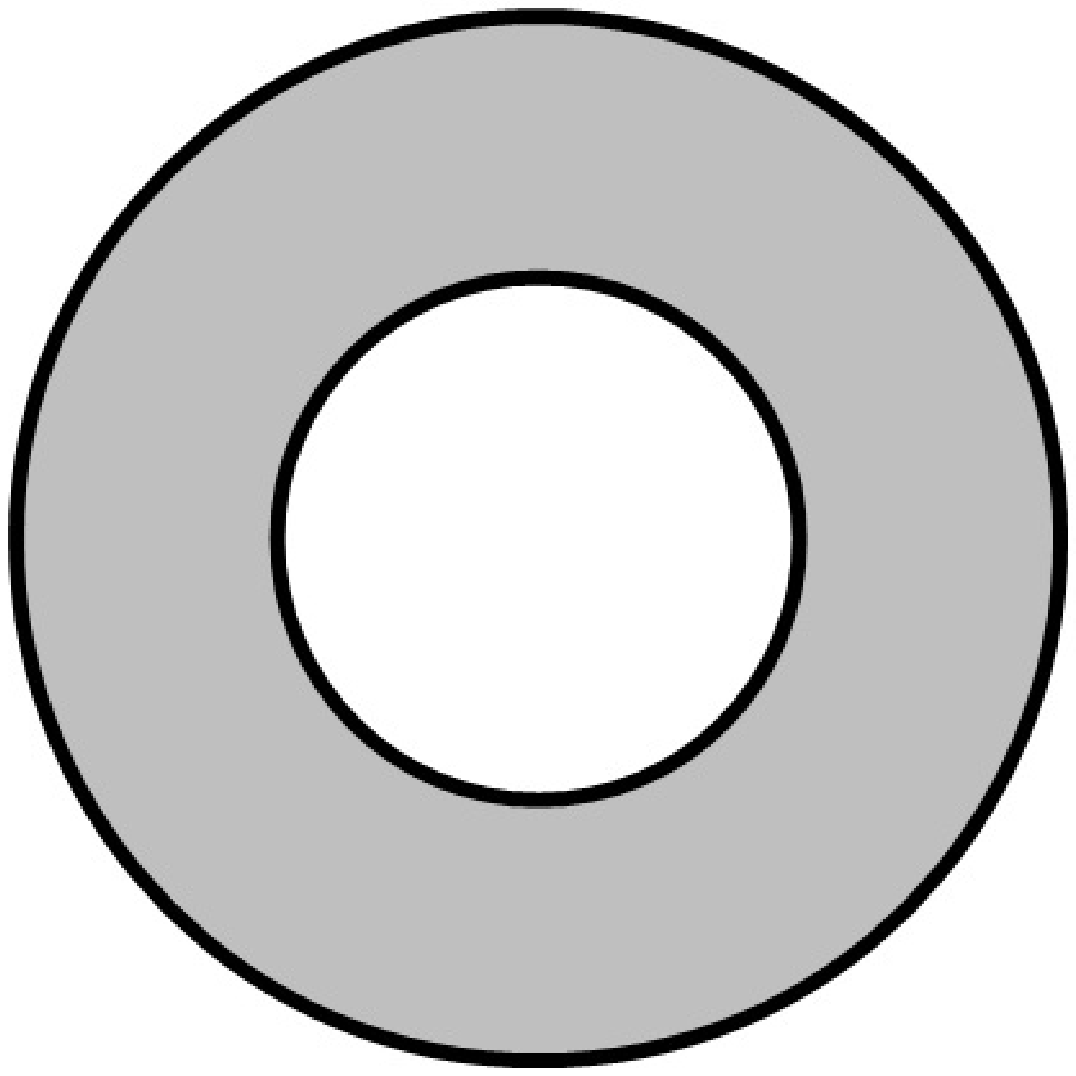}}
    \label{fig:mfld_orig}
}
\subfigure[]
{
  \centering
    {\includegraphics[scale=0.2]{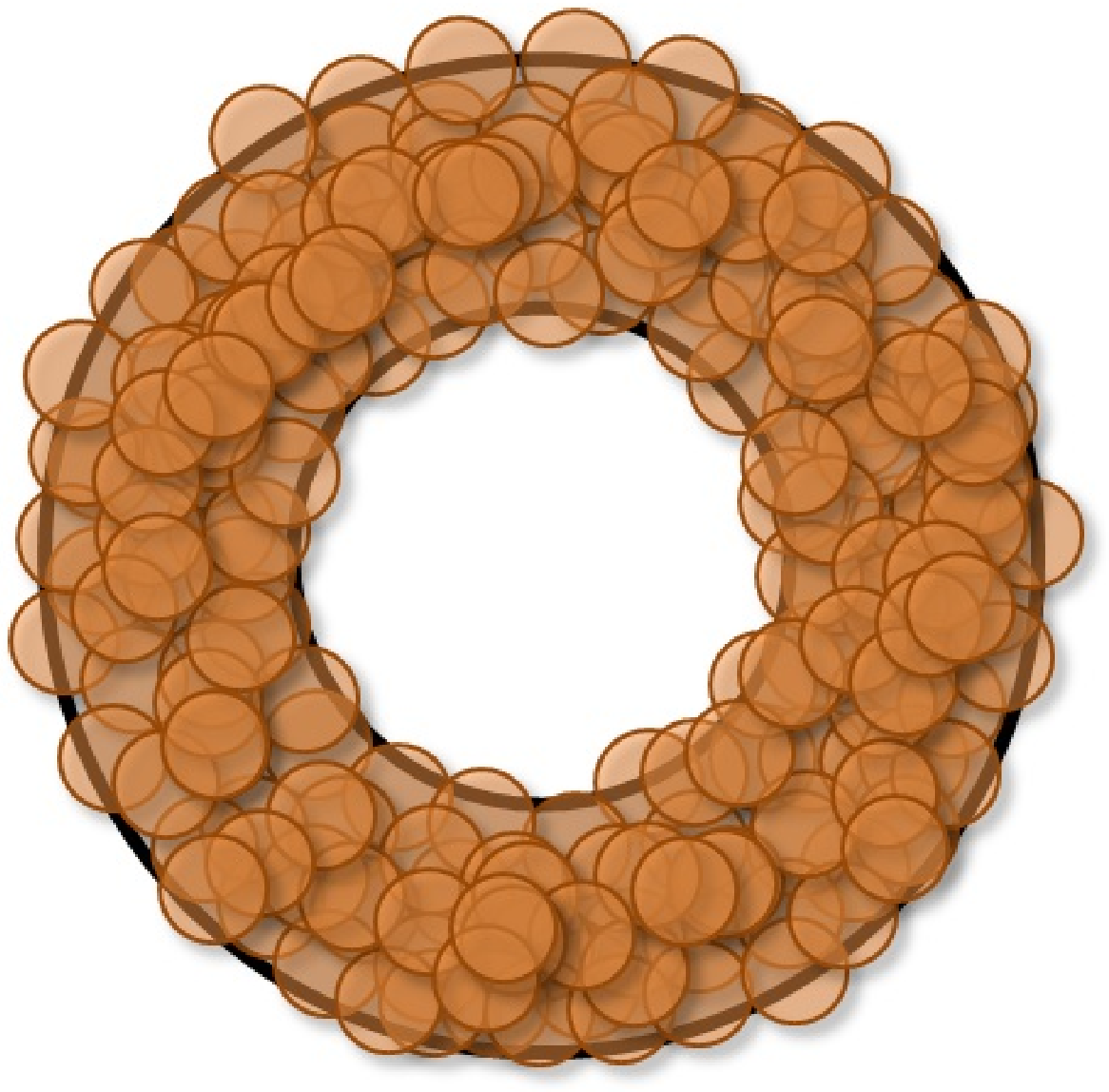}}
    \label{fig:mfld_no_noise}
}
\subfigure[]
{
  \centering
    {\includegraphics[scale=0.2]{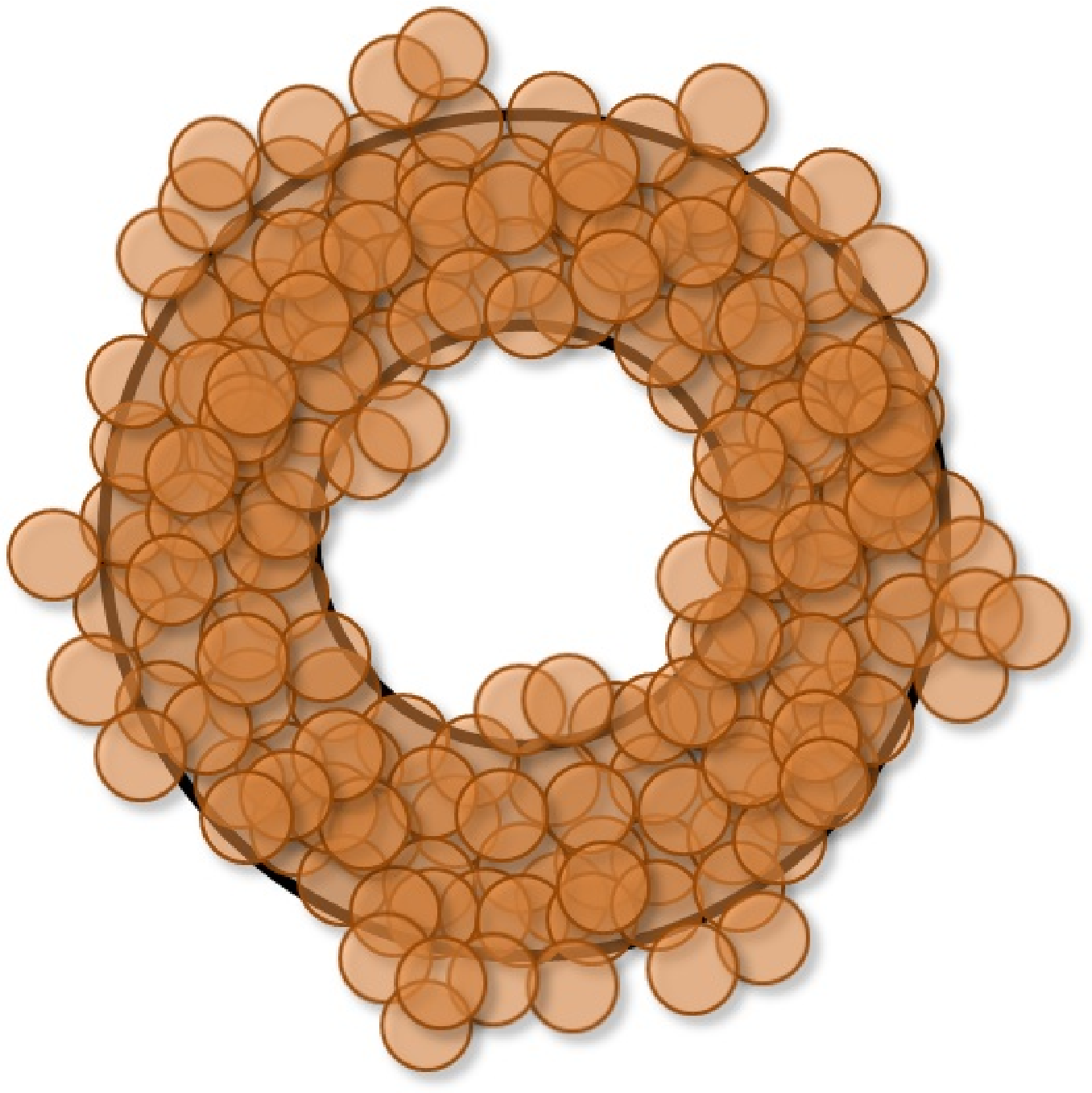}}
    \label{fig:mfld_noise_low}
}
\subfigure[]
{
  \centering
    {\includegraphics[scale=0.2]{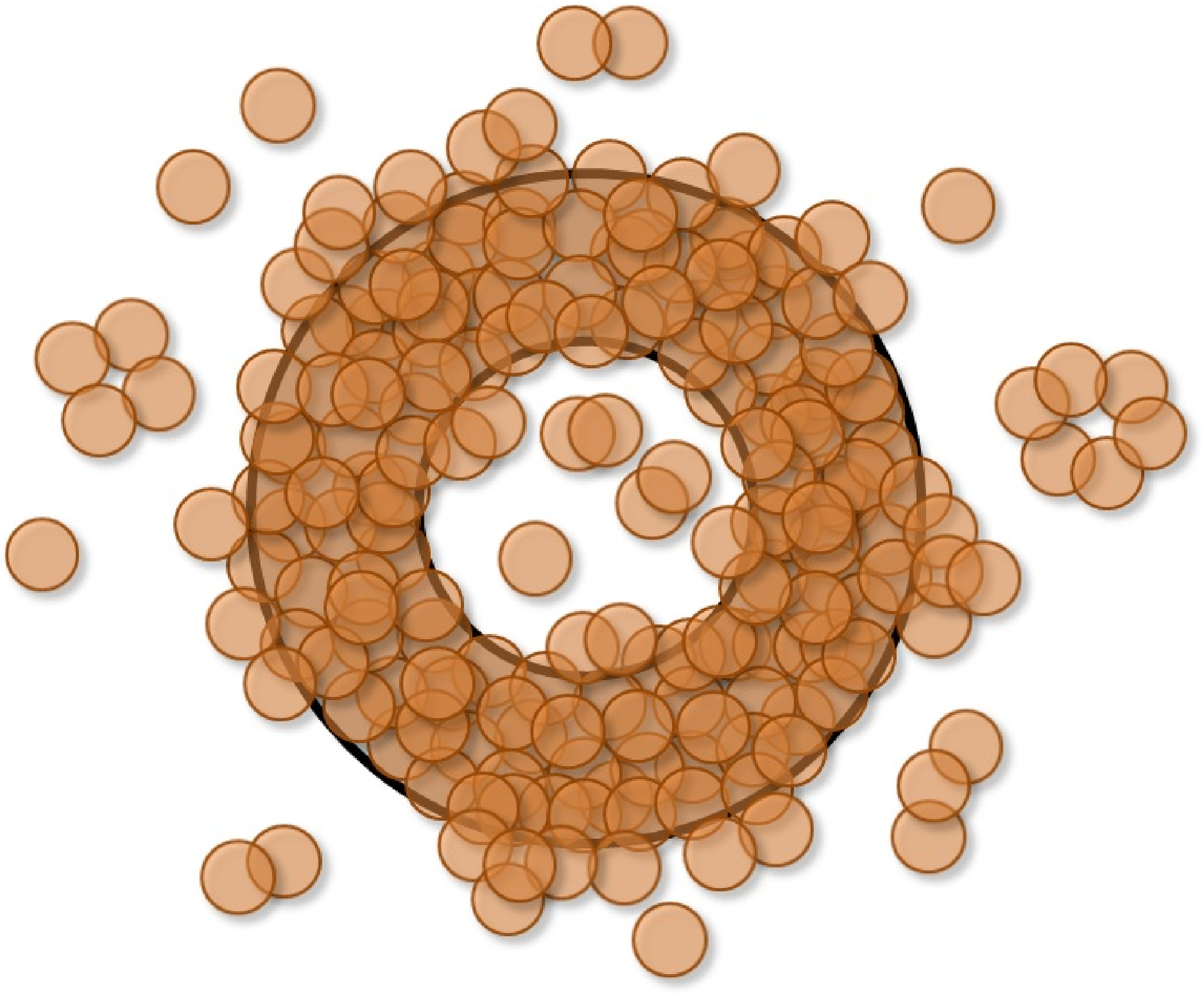}}
    \label{fig:mfld_noise_high}
}

\caption{(a) The original space  $\cM$ (an annulus) that we wish to recover from random samples. (b) With the appropriate choice of radius, we can easily recover the homology of the original space from random samples from $\cM$. (c) In the presence of bounded noise, homology recovery is undamaged. (d) In the presence of unbounded noise, many extraneous homology elements appear, and significantly interfere with homology recovery.}
\label{fig:complexes:simp_comp}
\end{figure}

In order to  be able, eventually, to extend the work in \cite{niyogi_topological_2011} beyond Gaussian noise, and make more concrete statements about the probabilistic features of the homology this extension generates, it is necessary to first focus on the behaviour of samples generated by pure noise, with no underlying manifold. In this case, thinking of the above setup,  the manifold $\cM$ is simply the point at the origin, and the homology that we shall be trying to recapture is trivial. Nevertheless, we shall see that differing noise models can make this task extremely delicate, regardless of sample size.

\subsection{Some sample results}
To start being more concrete, let $$\cX_n = \set{X_1,\ldots,X_n}$$ be a set of $n$ $\iid$ random samples in $\R^d$, from a common density function $f$. Recall that the
abstract simplicial complex  $\CC(\cX,\e)$  constructed according to the following rules is called the \cech complex associated to $\cX$ and $\e$:
\begin{enumerate}
\item The $0$-simplices  of  $\CC(\cX,\e)$   are the points in $\cX$,
\item An $n$-simplex $\sigma=[x_{i_0},\ldots,x_{i_n}]$ is in $\CC(\cX,\e)$ if $\bigcap_{k=0}^{n}
B_{x_{i_k}}(\e) \ne  \emptyset$,
\end{enumerate}
An important result, known as the `nerve theorem', links \cech complexes
 and the neighborhood set $U$ of \eqref{U:equn},  establishing that they are  homotopy equivalent  (cf.\ \cite{borsuk_imbedding_1948}).  In particular,  they have the same  Betti numbers, measures of homology that we shall concentrate on in what follows.

 If the sample distribution has a compact support $S$, then it is easy to show that,   for large enough $n$,
 \beqq
\CC(\cX,\e) \ \simeq \ \bigcup_{k=1}^n B_\e(X_k) \ \approx\  \tube(S,\e)  \ \teq \ \big\{x\in \R^d: \min_{y\in S} \|x-y\|\,\leq \, \e \big\},
\eeqq
where $\simeq$ denotes homotopy equivalence and $\norm{\cdot}$ is the standard $L^2$ norm in $\R^d$. Thus, there is not much to study in this case. However, when the support of the distribution is unbounded, interesting phenomena occur.

To study these phenomena, we shall consider three representative examples of probability densities. These are the  \emph{power-law}, \emph{exponential}, and the \emph{standard Gaussian} distributions,
whose density functions are given, respectively, by
\begin{align}
\label{eq:crackle:dist_pl}
\fpl(x) &\teq \frac{\cpl}{1+\norm{x}^\alpha},\\
\label{eq:crackle:dist_exp}
\fexp(x) &\teq \cexp e^{-\norm{x}},\\
\label{eq:crackle:dist_norm}
\fnorm(x) &\teq \cnorm e^{-\norm{x}^2/2},
\end{align}
where $\alpha > d$ and $\cpl,\cexp,\cnorm$ are  appropriate normalization constants that will not be of concern to us.

For large samples from any of these distributions we shall show that  there exists a `core' - a region in which the density of points is very high and so
placing unit balls around them completely covers the region. Consequently, the \cech complex inside the core is contractible. The size of the core obviously grows to infinity as the sample size $n$ goes to infinity, but its exact size will depend on the underlying distribution. For the three examples above, if we
denote the radius of the core by $\Rnc$, we shall prove in Section \ref{sec:crackle:core} that
\[
\Rnc \sim \begin{cases} (n/\log n)^{1/\alpha} & f(x) \propto \frac{1}{1+\norm{x}^\alpha} ,\\
\log n & f(x) \propto e^{-\norm{x}}, \\
\sqrt{2\log n}  & f(x) \propto e^{-\norm{x}^2/2}.
 \end{cases}
\]
Note that in all three cases we have tacitly assumed that the  cores are balls, a natural consequence of the spherical symmetry of the probability densities.

Beyond the core, the topology is more varied.
For fixed $n$,  there may be additional isolated components, but no longer enough placed densely  enough to connect with one another and to form  a contractible set. Indeed, we shall show that the individual components will typically have enough homology to be, individually, non-contractible.
Thus, in this region, the topology of the \cech complex is highly nontrivial, and many homology elements of different orders appear.
We call this phenomenon `crackling', akin to the well known phenomenon caused by noise interference in audio signals and commonly referred to as crackling.

As for core size, the exact crackling behaviour depends on the choice of distribution.  It turns out that  Gaussian samples do not lead to crackling, but the other two cases do. To describe this, with some imprecision of notation we shall write $[a,b)$ not only for an interval on the real line, but also for the annulus
 \beqq
 [a,b) \ \teq \  \big\{ x\in \R^d: a\leq \|x\| < b\big\}.
 \eeqq
 In Sections \ref{sec:crackle:pl} and \ref{sec:crackle:exp} we  shall show that the exterior of the core can be divided into disjoint spherical annuli at radii
$$R_n^c  \ll  R_{d-1,n} \ll R_{d-2,n} \ll \cdots \ll R_{0,n}$$
 (defined differently for each of the two crackling distributions) with different types of crackling (i.e.\ of homology) dominating in different regions.

 In $[R_{0,n},\infty)$ there are mostly disconnected points, and no structures with nontrivial homology.
 In $[R_{1,n},R_{0,n})$ connectivity is a bit higher, and a finite number of $1$-cycles appear. In $[R_{2,n},R_{1,n})$ we have a finite number of $2$-cycles, while the number of $1$-cycles grows to infinity as $n\to\infty$. In general, in $[R_{k,n},R_{k-1,n})$, as $n\to\infty$ we have a finite number of $k$-cycles, infinitely many $l$-cycles for $l<k$, and no cycles of dimension $l>k$. In other words, the crackle starts with a pure dust at $R_{n,0}$ and as we get closer to the core, higher dimensional homology  gradually appears.  See Figure \ref{fig:crackle:layers} in the following section
for more details.

As we already mentioned, the Gaussian distribution is fundamentally different than the other two, and does not lead to crackling. In Section \ref{sec:crackle:norm} we show that, for the Gaussian distribution, there are hardly any points located outside the core. Thus, as $n\to \infty$, the union of balls around the sample points becomes a giant contractible ball of radius of order $\sqrt{2\log n}$.

It is now possible to understand a little better how the results of this paper relate to  the noisy manifold learning problem discussed above. For example, if the distribution of the noise is Gaussian, our results imply that if the manifold is well behaved, and the sample size is moderate, noise  outliers should not significantly interfere with homology recovery, since Gaussian noise does not introduce artificial homology elements with large samples. However, there is a delicate counterbalance here between `moderate' and `large'. Once the sample size is large,  the core is also large, and the reconstructed manifold will have the topology of
$\cM\oplus B_{O(\sqrt{2\log n})}(0)$, where $\oplus$ is Minkowski addition. As $n$ grows, the core will eventually envelope any compact manifold, and thus the homology of $\cM$ will be hidden by that of the core.

 On the other hand, if the distribution of the noise is power-law or exponential, then noise outliers will typically generate extraneous homology elements that, for almost any sample size, will complicate the estimation of the original manifold.
 Furthermore, increasing the sample size in no way solves this problem.
 Note that this issue is in addition to the fact that increasing the sample size will, as in the Gaussian case,
 create the problem of a large core concealing the  topology
 of $\cM$.

 Thus, from a practical point of view, the message of this paper is that outliers cause problems in manifold estimation when  noise is present, a fact well known to all practitioners who have worked in the area. What is qualitatively new here is a quantification of how this happens, and how it relates to the distribution of the noise. We do not attempt {to solve this problem here, but unfortunately it follows from the results of this paper that algorithms for handling outliers will probably involve knowing at least the tail behaviour of the error distribution, despite the fact that in practical situations one does not generally want to take as known prior knowledge.

 \subsection{On persistence intervals}
 While the above discussion has concentrated on the persistence of noise induced crackle as sample sizes grow, and the regions in $\R^d$ in which different types of homology appear, the proofs below also yield information about the more classical persistence diagrams of topological data analysis (cf.\ \cite{carlsson_topology_2009,edelsbrunner_persistent_2008,edelsbrunner_book,ghrist_barcodes:_2008}).

  For example, in the two cases for which crackle persists -- the power-law and exponential cases  -- estimates of the type appearing in Section \ref{proofs:sec} indicate that, with high probability, there exist extremely long bars in the bar code representation of persistent homology. Up to lower order corrections, preliminary calculations show that  bar lengths for the $k$-th homology can be as large as  $O(n^{\alpha_k})$ for the power-law case, and  $\beta_k (\log \log n)$ for the exponential case, for appropriate $\alpha_k$ and $\beta_k$. More detailed studies of these phenomena will appear in a later publication.


\subsection{Poisson Processes}
\label{poisson:subsec}
Although we have described everything so far in terms of a random sample   $\cX$ of $n$ points taken from a density $f$, there is another way to approach the results of this paper, and that is to replace the points of $\cX$ with the points of a
$d$-dimensional Poisson process $\cP_n$ whose intensity function is given by $\lambda_n = n f$. In this case the number of points is no longer fixed, but has mean $n$.

All the results of this paper stated for $\cX$ hold, without any change, if we replace $\cX$ by $\cP$.

\subsection{Disclaimers}
Before starting the paper in earnest, and so as not to be accused of myopia, we note that the subject of manifold learning is obviously much broader that that described above, and  algorithms for `estimating' an underlying manifold from
a finite sample abound in the statistics and computer science literatures.
Very few of them, however, take an algebraic point of view that we or the literature quoted above take.
Furthermore, we note that other important results about the homology of Rips and \cech complexes for various distributions can be found in the papers \cite{kahle_random_2011,kahle_limit_2010} and  \cite{bobrowski_distance_2011} .
However, the methods and emphases of these papers are rather different.

\section{Results}
\label{sec:results}
In this section we shall present all our main results, along with some  discussion, more technical than that of the Introduction. Recall from Section \ref{poisson:subsec} that although we present all results for the point set $\cX$, they also hold if we replace the points of $\cX$ by the points of an appropriate Poisson process.
All proofs are deferred to Section \ref{proofs:sec}.

\subsection{The Core of Distributions with Unbounded Support} \label{sec:crackle:core}


We start by examining the core of the power-law, exponential and Gaussian distributions. These distributions are spherically symmetric and the samples are concentrated near the origin.
By `core' we refer to a centered ball $\Brn \teq B_{R_n}(0) \subset \R^d$ containing a very large number of points from the sample $\cX_n$, such that
\[
    \Brn \subset \bigcup_{X\in \cX_n \cap \Brn} B_1(X).
\]
i.e.\! the unit balls around the sample points completely cover $\Brn$.
In this case the homology of $\bigcup_{X\in\cX_n\cap \Brn} B_1(X)$, or equivalently, of $\CC(\cX_n \cap \Brn,1)$, is trivial.
Obviously, as $n\to\infty$, the radius $R_n$ grows as well.

Let $\set{R_n}_{n=1}^\infty$ be an increasing sequence of positive numbers.
Define by $C_n$ the event that $\Brn$ is covered, i.e.
\[
    C_n \teq \set{\Brn \subset \bigcup_{X\in \cX_n \cap \Brn} B_1(X)}.
\]
We wish to find the largest possible value of $R_n$ such that
$\prob{C_n} \to 1$.
The following theorem presents lower bounds for this value.


\begin{theorem}\label{thm:crackle:core} Let $\eps >0$, and define
\[
\Rnc \triangleq \begin{cases} \param{\frac{\dpl n}{\param{\log n - e^{-\eps} \log \log n}}-1}^{1/\alpha} & f = \fpl, \\
\log n - \log\log\log n -\dexp-\eps & f = \fexp, \\
\sqrt{2\param{\log n -\log\log\log n -\dnorm-\eps}}  & f = \fnorm,
 \end{cases}
\]
where the three distributions are given by \eqref{eq:crackle:dist_pl}--\eqref{eq:crackle:dist_norm}, and
\begin{align*}
\dpl &=  \cpl\alpha 2^{-d} d^{-(1+d/2)}, \\
\dexp &= (1+d/2)\log d +d\log 2-\log \cexp ,\\
\dnorm &= (1+d/2)\log d + (d-1)\log 2 -\log \cnorm.
\end{align*}
If $R_n \le \Rnc$, then
\[
\prob{C_n} \to 1.
\]
\end{theorem}

Theorem \ref{thm:crackle:core} implies that the core size has a completely different order of magnitude for each of the three distributions. The heavy-tailed, power-law distribution has the largest core, while the core of the Gaussian distribution is the smallest.
In the following sections we shall study the behaviour of the \cech complex outside the core.


\subsection{How Power-Law Noise Crackles} \label{sec:crackle:pl}


In this section we explore the crackling phenomenon in the power-law distribution $f=f_p$.
Let $\Brn\subset \R^d$ be the centered ball with radius $R_n$, and let
\[
    \CC_n \teq \CC(\cX_n \cap (\Brn)^c,1),
\]
be the \cech complex constructed from sample points outside $\Brn$.
We wish to study
\[
    \bk \teq \beta_k(\CC_n),
\]
the $k$-th Betti number of $\CC_n$.

Note that the minimum number of points required to form a $k$-dimensional cycle ($k\ge 1$) is $k+2$. For $k\ge 1$ and $\cY\subset \R^d$, denote
\[
    T_k(\cY) \teq \ind\set{\abs{\cY} = k+2,\ \beta_k(\CC(\cY,1)) = 1},
\]
i.e.\! $T_k$ takes the value $1$ if $\CC(\cY,1)$ is a minimal $k$-dimensional cycle, and $0$ otherwise. This indicator function will be used to define the limits of the Betti numbers.


\begin{theorem}\label{thm:crackle:mean_bk_pl}
If $\limninf n R_n^{-\alpha} = 0$, then
\begin{align*}
    \limninf \param{ n R_n^{d-\alpha} }^{-1}\mean{\beta_{0,n}} &= \mplo , \\
     \limninf\param{n^{k+2} R_n^{d- \alpha(k+2)}}^{-1}\mean{\bk} &= \mpl ,\quad 1 \le k \le d-1
\end{align*}
where
\begin{align}
\label{eq:crackle:mplo}
\mplo &\teq \frac{s_{d-1} \cpl}{\alpha-d}, \\
\label{eq:crackle:mpl}
\mpl &\teq \frac{s_{d-1}\cpl^{k+2}}{(\alpha(k+2)-d)(k+2)!}\int_{(\R^d)^{k+1}} T_k(0,\by)d\by, \quad 1\le k\le d-1,
\end{align}
and where $s_{d-1}$ is the surface area of the $(d-1)$-dimensional unit sphere in $\R^d$.
\end{theorem}


Next, we define the following values, which will serve as critical radii for the crackle,
\begin{align*}
\Roe &\teq n^{\param{\frac{1}{\alpha-d}+\eps}}, \\
\Ro &\teq \Ro^0 ,\\
\Rke &\teq n^{\param{\frac{1}{\alpha-d/(k+2)}+\eps}}, \quad (k\ge 1) \\
\Rk &\teq \Rk^0.
\end{align*}
The following is a straightforward corollary of Theorem \ref{thm:crackle:mean_bk_pl}, and summarizes the behaviour of $\mean{\bk}$ in the power-law case.


\begin{corollary}\label{cor:crackle:mean_bk_pl}
For $k\ge 0$ and $\eps >0$,
\[
    \limninf\mean{\bk} = \begin{cases}
    0 & R_n = \Rke, \\
    \mpl & R_n = \Rk ,\\
    \infty & R_n = \Rk^{-\eps},
    \end{cases}
\]
\end{corollary}


Theorem \ref{thm:crackle:mean_bk_pl} and Corollary \ref{cor:crackle:mean_bk_pl} reveal that the crackling behaviour is organized into separate `layers', see Figure \ref{fig:crackle:layers}. Dividing $\R^d$ into a sequence of annuli at radii
\[
    \Roe \gg \Ro \gg R_{1,n}^\eps \gg R_{1,n} \gg \cdots \gg R_{d-1,n}^\eps \gg R_{d-1,n} \gg \Rnc,
\]
we observe a different behaviour of the Betti numbers in each annulus.
We shall briefly review the behaviour in each annulus, in a decreasing order of radii values. The following description is mainly qualitative, and refers to expected values only.
\begin{itemize}
\item $[\Roe,\infty)$ - there are hardly any points ($\beta_k\sim 0$, $0\le k \le d-1$).
\item $[\Ro,\Roe)$ - points start to appear, and $\beta_0\sim \mplo$. The points are very few and scattered, so no cycles are generated ($\beta_k \sim 0$, $1\le k \le d-1$).
\item $[R_{1,n}^\eps,\Ro)$ - the number of components grows to infinity, but no cycles are formed yet ($\beta_0 \sim \infty$, and $\beta_k = 0$, $1 \le k \le d-1$).
\item $[R_{1,n},R_{1,n}^\eps)$ - a finite number of $1$-dimensional cycles show up, among the infinite number of components ($\beta_0 \sim \infty$, $\beta_1\sim\mu_{\mathrm{p},1}$, and $\beta_k = 0$, $1 \le k \le d-1$).
\item $[R_{2,n}^\eps,R_{1,n})$ - we have $\beta_0\sim\infty$, $\beta_1\sim\infty$, and $\beta_k\sim 0$ for $k\ge 1$.
\end{itemize}
This process goes on, until the $(d-1)$-dimensional cycles appear -
\begin{itemize}
\item $[R_{d-1},R_{d-1}^\eps)$ - we have $\beta_{d-1}\sim \mu_{\mathrm{p},d-1}$ and $\beta_k\sim\infty$ for $0\le k \le d-2$.
\item $[\Rnc,R_{d-1})$ - just before we reach the core, the complex exhibits the most intricate structure, with $\beta_k \sim \infty$ for $0\le k \le d-1$.
\end{itemize}

Note that there is a very fast phase transition as we move from the contractible core to the first crackle layer. At this point we do not know exactly where and how this phase transition takes place. A reasonable conjecture would be that the transition occurs at $R_n = n^{1/\alpha}$ (since at this radius the term $n R_n^{-\alpha}$ that appears in Theorem \ref{thm:crackle:mean_bk_pl} changes its limit, affecting the limiting Betti numbers). However, this remains for future work.

\begin{figure}[h]
\centering
  \includegraphics[scale=0.6]{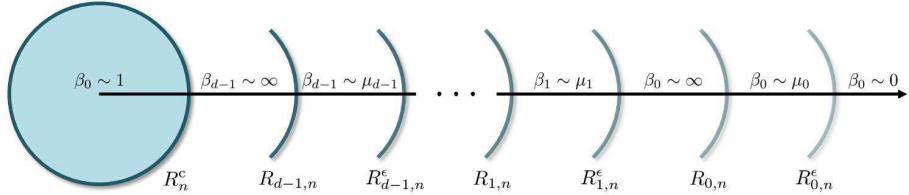}
\caption[Crackle layers]{The layered behaviour of crackle. Inside the core ($B_{\Rnc}$) the complex consists of a single component and no cycles. The exterior of the core is divided into separate annuli. Going from right to left, we see how the Betti numbers grow. In each annulus we present the Betti number that was most recently changed.
}
\label{fig:crackle:layers}
\end{figure}


\subsection{How Exponential Noise Crackles} \label{sec:crackle:exp}


In this section we focus on the exponential density function $f=\fexp$.
The results in this section are very similar to the those for the power law distribution, and we shall describe them briefly. Differences lie in the specific values of the $\Rk$ and in the terms in the limit formulae.


\begin{theorem}\label{thm:crackle:mean_bk_exp}
If $\limninf ne^{-R_n} = 0$, then
\begin{align*}
    \limninf \param{n R_n^{d-1} e^{-R_n}}^{-1}\mean{\beta_{0,n}} &= \mexpo , \\
     \limninf\param{n^{k+2} R_n^{d-1} e^{-(k+2)R_n}}^{-1} \mean{\bk} &= \mexp ,\ k\ge 1
\end{align*}
where
\begin{align}
\label{eq:crackle:mexpo}
\mexpo &\teq s_{d-1}\cexp,\\
\label{eq:crackle:mexp}
\mexp &\teq \frac{s_{d-1}\cexp^{k+2}}{(k+2)!} \int_0^\infty \int_{(\R^d)^{k+1}}T_k(0,\by) e^{-\param{(k+2)\rho + \sum_{i=1}^{k+1}y_i^1}} \prod_{i=1}^{k+1} \ind\set{y_i^1 > -\rho} d\by d\rho,
\end{align}
and where $y_i^1$ is the first coordinate of $y_i\in\R^d$.
\end{theorem}


Next, define
\begin{align*}
\Roe &\teq  \log n + \param{d-1+\eps}\log \log n, \\
\Ro &\teq \Ro^0 ,\\
\Rke &\teq \log n + \param{\frac{d-1}{k+2}+\eps}\log \log n, \quad (k\ge 1)\\
\Rk &\teq \Rk^0.
\end{align*}
From Theorem \ref{thm:crackle:mean_bk_exp} we can conclude the following.


\begin{corollary}\label{cor:crackle:mean_bk_exp}
For $k\ge 0$ and $\eps >0$,
\[
    \limninf\mean{\bk} = \begin{cases}
    0 & R_n = \Rke, \\
    \mexp & R_n = \Rk ,\\
    \infty & R_n = \Rk^{-\eps},
    \end{cases}
\]
\end{corollary}


As in the power-law case, Theorem \ref{thm:crackle:mean_bk_exp} implies the same `layered' behaviour, the only difference being in the values of $\Rk$. From examining the values of $\Rnc$, and $\Rk$ it is reasonable to guess that the phase transition in the exponential case occurs at $R_n = \log n$.


\subsection{Gaussian Noise Does Not Crackle} \label{sec:crackle:norm}


Simplicial complexes built over vertices sampled from
the standard Gaussian distribution exhibit a completely different behaviour to that we saw in the power-law and  exponential cases. Define
\[
\Roe \teq \sqrt{2\log n + (d-2+\eps)\log\log n},
\]
then
\begin{theorem}\label{thm:crackle:mean_bk_norm}
If $f=\fnorm$, $\eps > 0$, and $R_n = \Roe$, then for $0\le k \le d-1$
\[
\limninf \mean{\bk} = 0.
\]
\end{theorem}
Note that in the Gaussian case $\limninf\param{ \Roe - \Rnc} = 0$. This implies that as $n\to\infty$ we have the core which is contractible, and outside the core there is hardly anything. In other words, the ball placed around every new point we add to the sample immediately connects to the core, and thus, the Gaussian noise \textit{does not crackle}.


\section{Proofs}
\label{proofs:sec}
We now turn to proofs, starting with the proof of the main result  of Section \ref{sec:crackle:core}.

\subsection{The Core}


\begin{proof}[Proof of Theorem \ref{thm:crackle:core}]
The proof covers all three distributions, except for specific calculations near the end.
Take a grid on $\R^d$ of size $g = \frac{1}{2\sqrt{d}}$.
Let $\cQ_n$ be the collection of cubes in this grid that are contained in  $B_{R_n}$. Let $\tC_n$ be the following event
\[
    \tC_n \teq \set{\forall Q\in \cQ_n : Q\cap \cX_n \ne \emptyset},
\]
i.e.\! $\tC_n$ is the event that every cube in $\cQ_n$ contains at least one point from $\cX_n$. Recall the definition of $C_n$,
\[
    C_n \teq \set{\Brn \subset \bigcup_{X\in \cX_n \cap \Brn} B_1(X)}.
\]
Then it is easy to show that $\tC_n \subset C_n$. The complementary event $\tC_n^c$ is the event that at least one cube is empty. Thus,
\[
\probx{\tC_n^c} \le \sum_{Q\in \cQ_n} \prob{Q \cap \cX_n = \emptyset} = \sum_{Q\in \cQ_n}(1-p(Q))^n \le \sum_{Q\in \cQ_n}e^{-np(Q)}
\]
where
\[
p(Q) = \int_Q f(z)dz \ge  g^d f(R_n).
\]
In addition, the number of cubes that are contained in $B_{R_n}$ is less than $\param{2{{R_n}/{g}}}^d$.
Therefore,
\begin{equation}\label{eq:crackle:empty_box}
\probx{\tC_n^c} \le (2 g^{-1})^d R_n^d e^{-n g^d f(R_n) }.
\end{equation}
Now, choose any $\eps > 0$ and set
\[
R_n = \Rnc \triangleq \begin{cases} \param{\frac{\dpl n}{\param{\log n - e^{-\eps} \log \log n}}-1}^{1/\alpha} & f = \fpl, \\
\log n - \log\log\log n -\dexp-\eps & f= \fexp, \\
\sqrt{2\param{\log n -\log\log\log n -\dnorm-\eps}}  & f = \fnorm,
 \end{cases}
\]
where
\begin{align*}
\dpl &=  \cpl\alpha 2^{-d} d^{-(1+d/2)}, \\
\dexp &= \log d -\log \cexp - \log g^d, \\
\dnorm &= \log(d/2) -\log \cnorm - \log g^d.
\end{align*}
It is easy to verify that in all cases we have
\[
R_n^d e^{-n g^d f(R_n) }  \to 0.
\]
Thus, from \eqref{eq:crackle:empty_box} we conclude that $\probx{\tC_n} \to 1$. Since $\prob{C_n} \ge \probx{\tC_n}$ we now have that for $R_n = \Rnc$, in each of the distributions,
\[
\prob{C_n} \to 1,
\]
which completes the proof.
\end{proof}


\subsection{Crackle - Notation and General Lemmas}


For $R_n > 0$, set
\[
    \cxr \teq \cX_n \cap (\Brn)^c,
\]
i.e.\! $\cxr$ consists of the points of $\cX_n$ located outside the ball $\Brn$. Next, recall the definition of $T_k$,
\[
    T_k(\cY) \teq \ind\set{\abs{\cY} = k+2,\ \beta_k(\CC(\cY,1)) = 1},
\]
for $\cY \subset \R^d$, and write
\begin{align*}
    \So & \teq \abs{\cxr}, \\
    \Sot &\teq \#\set{X \in \cxr : X \textrm{ is a connected component of } \CC(\cX_n,1)}\\
    \Sk &\teq \sum_{\cY \subset \cxr} T_k(\cY),\\
    \Skt &\teq \sum_{\cY \subset \cxr} T_k(\cY)\ind\set{\CC(\cY,1) \textrm{ is a connected component of } \CC(\cX_n,1)},\\
    \Lk &\teq \sum_{\cY \subset \cxr} \ind\set{\abs{\cY} = k+3,\ \CC(\cY,1) \textrm{ is connected}},
\end{align*}
where $k\ge 1$. Observe that
\begin{align}
\label{eq:crackle:betti_0_ineq}
\Sot \le &\beta_{0,n} \le \So \\
\label{eq:crackle:betti_k_ineq}
    \Skt \le &\bk \le \Skt + \Lk,\quad k\ge 1
\end{align}
We will evaluate the limits of $\mean{\Sk}$, $\meanx{\Skt}$ and $\mean{\Lk}$ and deduce from these the limit of $\mean{\bk}$.

In addition,  set
\begin{align*}
    \eone &\teq (1,0,\ldots,0) \in \R^d ,\\
    f(r) &\teq f(r \eone), \ \ r \in \R, \\
    U(\bx) &\teq \bigcup_{i=1}^{k} B_2(x_i), \ \ \bx\in (\R^d)^k, \\
    p(\bx) &\teq \int_{U(\bx)}f(z)dz,\ \ \bx\in (\R^d)^k.
\end{align*}
The following two lemmas are purely technical, but will considerably simplify our computations later.


\begin{lemma}\label{lem:crackle:symmetry_zero}
Let  $f:\R^d\to\R$ be a spherically symmetric probability density.
Then,
\begin{align*}
\mean{\So} &=  s_{d-1}n \int_{R_n}^\infty r^{d-1}f(r)dr ,\\
\meanx{\Sot} &=  s_{d-1}n \int_{R_n}^\infty r^{d-1}f(r)(1-np(r\eone))^{n-1}dr,
\end{align*}
where $s_{d-1}$ is the volume of the $d-1$ dimensional unit sphere.
\end{lemma}


\begin{proof}
$\So$ is simply a sum of Bernoulli variables, therefore
\[
\mean{\So}  = n \prob{\norm{X} > R_n} = n\int_{\R^d} f(x)\ind\set{\norm{x}>R_n}dx.
\]
Writing the integral in polar coordinates yields
\[
\mean{\So} = n \int_{R_n}^\infty \int_{S^{d-1}} f(r\theta)r^{d-1}J(\theta)d\theta dr,
\]
where $J(\theta) = \abs{\frac{\partial x}{\partial \theta}}$.
Since $f$ is spherically symmetric, $f(r\theta) = f(r)$, and therefore
\[
\mean{\So} = s_{d-1} n \int_{R_n}^\infty r^{d-1} f(r) dr.
\]
The proof for $\Sot$ is similar, using the fact that the probability that a point $x\in\R^d$ is disconnected from the rest of the complex $\CC(\cX_n,1)$ is $(1-p(x))^{n-1}$.


\end{proof}


\begin{lemma}\label{lem:crackle:symmetry}
Let $f:\R^d\to\R$ be a spherically symmetric probability density.
Then, for $k\ge 1$,
\begin{align*}
\mean{\Sk} &=  s_{d-1} \binom{n}{k+2} \int_{R_n}^\infty r^{d-1}f(r)G_k(r)dr ,\\
\meanx{\Skt} &= s_{d-1} \binom{n}{k+2}\int_{R_n}^\infty r^{d-1}f(r)\hat{G}_k(r)dr,
\end{align*}
where $s_{d-1}$ is the volume of the $d-1$ dimensional sphere, and where
\begin{align*}
    G_k(r) &\triangleq \int_{(\R^d)^{k+1}} f(\norm{r\eone+\by}) T_k(0,\by) \prod_{i=1}^{k+1} \ind\set{\norm{r\eone+y_i} > R_n} d\by, \\
\hat{G}_k(r) &\triangleq \int_{(\R^d)^{k+1}} f(\norm{r\eone+\by}) T_k(0,\by) \prod_{i=1}^{k+1} \ind\set{\norm{r\eone+y_i} > R_n} \\
&\times (1-p(r\eone, r\eone+\by))^{n-k-2}d\by.
\end{align*}

\end{lemma}


\begin{proof}
The proof is in the same spirit of the proof of Lemma \ref{lem:crackle:symmetry_zero}, but technically more complicated.
Thinking of $\Sk$ as a sum of Bernoulli variables, we have that
\[
\mean{\Sk} =  \binom{n}{k+2}\int_{(\R^d)^{k+2}}f(\bx) T_k(\bx) \prod_{i=1}^{k+2} \ind\set{\norm{x_i} > R_n} d\bx.
\]
Let $I_k$ denote the integral above. Then, using the change of variables
\begin{align*}
    x_1 &\to x, \qquad x_i \to x+ y_{i-1} ,\ \  (i>1),
\end{align*}
yields
\begin{align*}
I_k &= \int_{\norm{x}\ge R_n}  \int_{(\R^d)^{k+1}} f(x)f(x+\by) T_k(x,x+\by) \prod_{i=1}^{k+1} \ind\set{\norm{x+y_i} > R_n}d\by dx \\
&=\int_{\norm{x}\ge R_n}  \int_{(\R^d)^{k+1}} f(x)f(x+\by) T_k(0,\by) \prod_{i=1}^{k+1} \ind\set{\norm{x+y_i} > R_n}d\by dx.
\end{align*}
Moving to polar coordinates yields
\begin{align*}
I_k &= \int_{R_n}^\infty \int_{S^{d-1}}  \int_{(\R^d)^{k+1}} f(r\theta)f(r\theta+\by) T_k(0,\by) \\
&\times\prod_{i=1}^{k+1} \ind\set{\norm{r\theta+y_i} > R_n} r^{d-1}J(\theta)d\by d\theta dr\\
&= \int_{R_n}^\infty r^{d-1}f(r) \int_{S^{d-1}}  J(\theta)\int_{(\R^d)^{k+1}} f(\norm{r\theta+\by}) T_k(0,\by) \\
&\times \prod_{i=1}^{k+1} \ind\set{\norm{r\theta+y_i} > R_n} d\by d\theta dr,
\end{align*}
where $J(\theta) = \abs{\frac{\partial x}{\partial \theta}}$, and $f(x) = f(\norm{x})$ by the spherical symmetry assumption.
Set
\[
    G_k(r,\theta) \triangleq \int_{(\R^d)^{k+1}} f(\norm{r\theta+\by}) T_k(0,\by) \prod_{i=1}^{k+1} \ind\set{\norm{r\theta+y_i} > R_n} d\by.
\]
Since $T_k$ is rotation invariant, it is easy to show that for every $\theta \in S^{d-1}$
\[
    G_k(r,\theta) = G_k(r,\eone) \triangleq G_k(r).
\]
Thus,
\begin{equation}\label{eq:crackle:I_k}
    I_k = s_{d-1}\int_{R_n}^\infty r^{d-1}f(r)G_k(r)dr.
\end{equation}
This completes the proof for $\Sk$. The proof for $\Skt$ is similar.


\end{proof}

In what follows, we shall use the following elementary limits:
\begin{enumerate}
\item For every $k > 0$,
\begin{equation} \label{eq:binom_limit}
    \limninf n^{-k} \binom{n}{k}  = \frac{1}{k!}
\end{equation}
\item For every sequence $a_n\to 0$ and $k\ge0$,
\begin{equation}\label{eq:power_exp}
    \limninf \frac{(1-a_n)^{n-k}}{e^{-na_n}} = 1
\end{equation}
\end{enumerate}


\subsection{Crackle - The Power Law Distribution}


In this section we prove the results in Section \ref{sec:crackle:pl}. First, we need a few lemmas.


\begin{lemma}\label{lem:crackle:pl_dust}
If $f=\fpl$, and $R_n\to\infty$, then
\[
\limninf \param{n R_n^{d-\alpha}}^{-1} \mean{\So} = \mplo,
\]
where $\mplo$ is defined in \eqref{eq:crackle:mplo}.

If, in addition, $nR_n^{-\alpha}\to 0$, then
\[
\limninf \param{ n R_n^{d-\alpha} }^{-1}\meanx{\Sot} = \mplo.
\]
\end{lemma}


\begin{proof}
From Lemma \ref{lem:crackle:symmetry_zero} we have that
\[
    \mean{\So} = s_{d-1}n \int_{R_n}^\infty r^{d-1}f(r)dr.
\]
Making the change of variables $r\to R_n \rho$ yields
\begin{align*}
    \mean{\So} &= s_{d-1}n  \int_1^\infty \frac{\cpl(R_n\rho)^{d-1}}{1+ (R_n\rho)^\alpha}R_n d\rho  \\
    &= s_{d-1}\cpl n R_n^{d-\alpha} \int_1^\infty \frac{\rho^{d-1}}{R_n^{-\alpha}+ \rho^\alpha} d\rho.
\end{align*}
Applying the dominated convergence theorem to the previous integral gives
\[
\limninf \param{nR_n^{d-\alpha}}^{-1}\mean{\So} = s_{d-1}\cpl\int_1^\infty \rho^{d-1-\alpha}d\rho = \frac{s_{d-1}\cpl}{\alpha-d} = \mplo.
\]
This proves the first part of the lemma.

Next, from Lemma \ref{lem:crackle:symmetry_zero} we have that
\[
    \meanx{\Sot} = s_{d-1}n \int_{R_n}^\infty r^{d-1}f(r)(1-p(r\eone))^{n-1}dr.
\]
The power term is bounded by $1$ and therefore will not affect the conditions needed for dominated convergence. Thus, using \eqref{eq:power_exp}, we only need to evaluate its limit.
\[
p(r\eone) = \int_{B_2(r\eone)}f(z)dz = \int_{B_2(0)}\frac{\cpl}{1+\norm{r\eone+z}}dz,
\]
and after the change of variables $r\to R_n\rho$ we have,
\[
p(R_n\rho\eone)  =\cpl R_n^{-\alpha}\int_{B_2(0)} \frac{1}{R_n^{-\alpha}+\norm{\rho\eone+R_n^{-1}z}^\alpha}dz.
\]
If $nR_n^{-\alpha}\to 0$, then, by dominated convergence, we have
\[
\limninf np(R_n\rho\eone) =0.
\]
Thus,
\[
        \limninf (1-p(R_n\rho\eone))^{n-1} = \limninf e^{-np(R_n\rho\eone)} =1,
\]
and therefore we have
\[
    \limninf\param{nR_n^{d-\alpha}}^{-1}\meanx{\Sot} = \limninf\param{nR_n^{d-\alpha}}^{-1}\mean{\So} = \mplo.
\]
This completes the proof of the second part of the lemma.
\end{proof}


\begin{lemma}\label{lem:crackle:pl_holes}
If $f=\fpl$, and $R_n\to\infty$ then
\[
\limninf \param{n^{k+2} R_n^{d- \alpha(k+2)}}^{-1} \mean{\Sk} = \mpl,
\]
where $\mpl$ is defined in \eqref{eq:crackle:mpl}.
If, in addition,$n R_n^{-\alpha} \to 0$, then
\[
\limninf \param{n^{k+2} R_n^{d-\alpha(k+2)}}^{-1} \meanx{\Skt} = \mpl.
\]
\end{lemma}


\begin{proof}
The proof is in the spirit of the proof of Lemma \ref{lem:crackle:pl_dust}, but technically more complicated.
From Lemma \ref{lem:crackle:symmetry} we have that
\[
\mean{\Sk} = \binom{n}{k+2} I_k,
\]
where
\[
    I_k = s_{d-1}\int_{R_n}^\infty r^{d-1}f(r)G_k(r)dr.
\]
Making the change of variables $r \to R_n\rho$ yields
\begin{align*}
I_k &= s_{d-1}R_n\int_1^\infty (R_n\rho)^{d-1} f(R_n\rho)  G_k(R_n\rho) d\rho \\
&= s_{d-1}\cpl^{k+2}(R_n)^{d-\alpha(k+2)}\int_1^\infty \int_{(\R^d)^{k+1}} \frac{\rho^{d-1}}{R_n^{-\alpha} + \rho^\alpha}\prod_{i=1}^{k+1}\frac{1}{R_n^{-\alpha} +\norm{\rho\eone + R_n^{-1}y_i}^\alpha} \\
&\ \ \times T_k(0,\by) \prod_{i=1}^{k+1} \ind\set{\norm{\rho\eone+ R_n^{-1}y_i} > 1} d\by.
\end{align*}
Thus, using \eqref{eq:binom_limit},
\begin{align*}
&(n^{k+2} R_n^{d-\alpha (k+2)})^{-1}\mean{\Sk} = \frac{s_{d-1}\cpl^{k+2}}{(k+2)!}\int_1^\infty \int_{(\R^d)^{k+1}} \frac{\rho^{d-1}}{R_n^{-\alpha} + \rho^\alpha}\\
&\quad \times T_k(0,\by) \prod_{i=1}^{k+1} \frac{1}{R_n^{-\alpha} +\norm{\rho\eone + R_n^{-1}y_i}^\alpha}\ind\set{\norm{\rho\eone+ R_n^{-1}y_i} > 1} d\by.
\end{align*}
It is easy to show that the integrand is bounded by an integrable term, so the dominated convergence theorem applies, yielding
\begin{align*}
&\limninf (n^{k+2} R_n^{d-\alpha (k+2)})^{-1}\mean{\Sk}\\
 & \qquad \qquad= \frac{s_{d-1}\cpl^{k+2}}{(k+2)!}\int_1^{\infty}\rho^{d-1-\alpha(k+2)}d\rho\int_{(\R^d)^{k+1}} T_k(0,\by)d\by \\
 & \qquad \qquad= \frac{s_{d-1}\cpl^{k+2}}{(\alpha(k+2)-d)(k+2)!}\int_{(\R^d)^{k+1}} T_k(0,\by)d\by\\
 & \qquad \qquad=\mpl.
\end{align*}
This proves the first part of the lemma.

Next, the terms $G_k(r)$ and $\hat{G}_k(r)$ in Lemma \ref{lem:crackle:symmetry} differ only by the term $(1-p(r\eone, r\eone+\by))^{n-k-2}$, so dominated convergence still applies.
Now,
\[
p(r\eone, r\eone+\by) = \int_{U(r\eone,r\eone+\by)} f(z)dz = \int_{U(0,\by)}f(r\eone+z)dz,
\]
and substituting $r\to R_n\rho$ yields,
\[
p(R_n\rho\eone , R_n\rho\eone+\by) = \cpl R_n^{-\alpha}\int_{U(0,\by)} \frac{1}{R_n^{-\alpha}+\norm{\rho\eone+R_n^{-1}z}^\alpha}dz.
\]
If $nR_n^{-\alpha}\to 0$, then using the dominated convergence we have
\[
\limninf np(R_n\rho\eone,R_n\rho\eone + \by) = 0.
\]
Thus,
\[
    \limninf e^{-np(R_n\rho\eone, R_n\rho\eone + \by)} =1,
\]
and therefore, using \eqref{eq:power_exp},
\begin{align*}
\limninf \param{n^{k+2} R_n^{d-\alpha (k+2)}}^{-1}\meanx{\Skt}  &= \limninf \param{n^{k+2} R_n^{d-\alpha (k+2)}}^{-1}\meanx{\Sk}  \\
&= \mpl.
\end{align*}
This completes the proof of the second part of the lemma.
\end{proof}


\begin{lemma}\label{lem:crackle:pl_other}
If $f=\fpl$, and $R_n\to\infty$ then
\[
\limninf \param{n^{k+3} R_n^{d-\alpha(k+3)}}^{-1} \mean{\Lk} = \mplh,
\]
for some $\mplh > 0$.
\end{lemma}


\begin{proof}
The proof is very similar to the proof of Lemma \ref{lem:crackle:pl_holes}. We need only replace $T_k$ with an indicator function that tests whether a  sub-complex generated by $k+3$ points is connected. The exact value of $\mplh$ will not be needed anywhere.
\end{proof}


We can now prove Theorem \ref{thm:crackle:mean_bk_pl}.


\begin{proof}[Proof of Theorem \ref{thm:crackle:mean_bk_pl}]
To prove the limit for $\beta_{0,n}$ simply combine Lemma \ref{lem:crackle:pl_dust} with the inequality \eqref{eq:crackle:betti_0_ineq}.
To prove the limit for $\bk$, $k\ge 1 $,  combine Lemmas \ref{lem:crackle:pl_holes} and \ref{lem:crackle:pl_other} with the inequality  \eqref{eq:crackle:betti_k_ineq}.
\end{proof}


%


\subsection{Crackle - The Exponential Distribution}


In this section we wish to prove Theorem \ref{thm:crackle:mean_bk_exp}. We start with the following lemmas.


\begin{lemma}\label{lem:crackle:exp_dust}
If $f=\fexp$, and $R_n\to\infty$ then,
\[
\limninf \param{n R_n^{d-1} e^{-R_n}}^{-1} \mean{\So} = \mexpo,
\]
where $\mexpo$ is defined in \eqref{eq:crackle:mexpo}.

If, in addition, $ne^{-R_n}\to 0$ then,
\[
\limninf \param{n R_n^{d-1} e^{-R_n}}^{-1} \meanx{\Sot} = \mexpo.
\]
\end{lemma}


\begin{proof}
From Lemma \ref{lem:crackle:symmetry_zero} we have that
\[
    \mean{\So} = s_{d-1}n \int_{R_n}^\infty r^{d-1}f(r)dr.
\]
Using the change of variables $r\to \rho + R_n$ yields
\begin{align*}
    \mean{\So} &= s_{d-1}n  \int_0^\infty (\rho+R_n)^{d-1}\cexp e^{-(\rho+R_n)} d\rho  \\
    &= s_{d-1}\cexp n R_n^{d-1}e^{-R_n} \int_0^\infty \param{\frac{\rho}{R_n}+1}^{d-1}e^{-\rho} d\rho.
\end{align*}
Applying dominated convergence to the last integral yields,
\[
\limninf \param{n R_n^{d-1} e^{-R_n}}^{-1} \mean{\So} = s_{d-1}\cexp\int_0^\infty e^{-\rho}d\rho = s_{d-1}\cexp = \mexpo.
\]
This proves the first part of the lemma.

Next, from Lemma \ref{lem:crackle:symmetry_zero} we have that
\[
    \meanx{\Sot} = s_{d-1}n \int_{R_n}^\infty r^{d-1}f(r)(1-p(r\eone))^{n-1}dr.
\]
The power term will not affect the dominated convergence conditions. Thus, we only need to evaluate its limit.
\[
p(r\eone) = \int_{B_2(r\eone)}f(z)dz = \int_{B_2(0)}\cexp e^{-\norm{r\eone+z}}dz,
\]
and after the change of variables $r\to \rho+R_n$ we have,
\[
p((\rho+R_n)\eone) = \int_{B_2(0)} \cexp e^{-\norm{(\rho+R_n)\eone + z}}dz \le e^{-(R_n+\rho)} \int_{B_2(0)} \cexp e^{\norm{z}}dz.
\]
If $ne^{-R_n}\to 0$, then
\[
\limninf np((\rho+R_n)\eone) =0.
\]
Thus,
\[
    \limninf e^{-np((\rho+R_n)\eone)} =1,
\]
and therefore, using \eqref{eq:power_exp}, we have
\[
   \limninf \param{n R_n^{d-1} e^{-R_n}}^{-1} \meanx{\Sot} = \limninf \param{n R_n^{d-1} e^{-R_n}}^{-1} \meanx{\So} = \mexpo.
\]
This completes the proof of the second part of the lemma.
\end{proof}


\begin{lemma}\label{lem:crackle:exp_holes}
If $f=\fexp$, and $R_n\to\infty$ then,
\[
\limninf \param{n^{k+2} R_n^{d-1} e^{-(k+2)R_n}}^{-1} \mean{\Sk} = \mexp,
\]
where $\mexp$ is defined in \eqref{eq:crackle:mexp}.

If, in addition, $ne^{-R_n}\to 0$ then,
\[
\limninf \param{n^{k+2} R_n^{d-1} e^{-(k+2)R_n}}^{-1} \meanx{\Skt} = \mexp.
\]
\end{lemma}


\begin{proof}
From Lemma \ref{lem:crackle:symmetry} we have that
\[
\mean{\Sk} = \frac{n^{k+2}}{(k+2)!} I_k,
\]
where
\[
    I_k = s_{d-1}\int_{R_n}^\infty r^{d-1}f(r)G_k(r)dr.
\]
Making the change of variables $r \to \rho + R_n$ yields
\begin{align*}
    I_k &= s_{d-1}\int_0^\infty (\rho+R_n)^{d-1}f(\rho+R_n)G_k(\rho+R_n)d\rho \\
    &= s_{d-1}\cexp^{k+2} \int_0^\infty \int_{(\R^d)^{k+1}} (\rho+R_n)^{d-1} e^{-(\rho+R_n)} \prod_{i=1}^{k+1} e^{-\norm{(\rho+R_n)\eone + y_i}} \\
    &\quad\times T_k(0,\by) \prod_{i=1}^{k+1} \ind\set{\norm{(\rho+R_n)\eone+y_i} > R_n} d\by d\rho \\
    &= s_{d-1}\cexp^{k+2}e^{-(k+2)R_n}R_n^{d-1} \int_0^\infty \int_{(\R^d)^{k+1}} \param{\frac{\rho}{R_n}+1}^{d-1} e^{-\rho} \\
    &\quad\times T_k(0,\by) \prod_{i=1}^{k+1} e^{-\norm{(\rho+R_n)\eone + y_i}} e^{R_n} \ind\set{\norm{(\rho+R_n)\eone+y_i} > R_n} d\by d\rho.
\end{align*}
The last integral can be easily shown to satisfy the  conditions of the dominated convergence theorem.
In addition, it is easy to show that
\[
    \limninf  e^{-\norm{(\rho+R_n)\eone +y_i}}e^{R_n}  = e^{-\param{\rho +\iprod{\eone,y_i}}} = e^{-(\rho + y_i^1)},
\]
where $y_i^1$ is the first coordinate of $y_i \in \R^d$, and also that
\[
    \limninf\ind\set{\norm{(\rho+R_n)\eone+y_i} > R_n} = \ind\set{y_i^1 \ge -\rho}.
\]
Altogether, we have that
\begin{align*}
    &\limninf \param{n^{k+2} R_n^{{d-1}} e^{-(k+2)R_n}}^{-1}  \mean{\Sk} \\
    &= \frac{s_{d-1}\cexp^{k+2}}{(k+2)!} \int_0^\infty \int_{(\R^d)^{k+1}} T_k(0,\by) e^{-\param{(k+2)\rho + \sum_{i=1}^{k+1}y_i^1}} \prod_{i=1}^{k+1} \ind\set{y_i^1 \ge -\rho} d\by d\rho,
\end{align*}
proving the first part of the lemma.

Next, as in the proof of Lemma \ref{lem:crackle:pl_holes}, we need to evaluate the term $p(r\eone, r\eone+\by)$.
\[
p(r\eone, r\eone+\by) = \int_{U(0,\by)}\cexp e^{-\norm{r\eone+z}}dz \le  \int_{U(0,\by)}\cexp e^{-(r-\norm{z})}dz.
\]
The change of variables $r\to\rho+R_n$ yields
\[
p((\rho+R_n)\eone,(\rho+R_n)\eone+\by) \le e^{-R_n}e^{-\rho} \int_{U(0,\by)}\cexp e^{\norm{z}}dz.
\]
If $ne^{-R_n}\to 0$, then
\[
\limninf n p((\rho+R_n)\eone,(\rho+R_n)\eone+\by) = 0.
\]
Thus,
\[
    \limninf e^{-np(R_n\rho\eone, R_n\rho\eone + \by)} =1,
\]
and therefore,
\begin{align*}
&\limninf \param{n^{k+2} R_n^{{d-1}} e^{-(k+2)R_n}}^{-1}  \meanx{\Skt} \\
&\qquad \qquad = \limninf \param{n^{k+2} R_n^{{d-1}} e^{-(k+2)R_n}}^{-1}  \mean{\Sk}= \mexp.
\end{align*}
This completes the proof.
\end{proof}


\begin{lemma}\label{lem:exp_other}
If $f=\fexp$, and $R_n\to\infty$ then
\[
\limninf \param{n^{k+3} R_n^{d-1} e^{-(k+3)R_n}}^{-1} \mean{\Lk} = \mexph.
\]
where $\mexph > 0$.
\end{lemma}


\begin{proof}
As for the proof of Lemma \ref{lem:crackle:pl_other}, mimic now  the proof of Lemma \ref{lem:crackle:exp_holes}, replacing $T_k$ with an indicator function that tests whether a sub-complex generated by $k+3$ points is connected.
\end{proof}


\begin{proof}[Proof of Theorem \ref{thm:crackle:mean_bk_exp}]
The proof follows the same steps as the proof of Theorem \ref{thm:crackle:mean_bk_pl}.
\end{proof}


\subsection{Crackle - The Gaussian Distribution}


In this section we  prove Theorem \ref{thm:crackle:mean_bk_norm}.


\begin{proof}[Proof of Theorem \ref{thm:crackle:mean_bk_norm}]
From Lemma \ref{lem:crackle:symmetry_zero} we have that
\[
    \mean{\So} = s_{d-1}n \int_{R_n}^\infty r^{d-1}f(r)dr.
\]
Making the change of variables $r \to (\rho^2 + R_n^2)^{1/2}$ which implies $dr = \frac{\rho}{(\rho^2+R_n^2)^{1/2}}d\rho$, we have
\begin{align*}
\mean{\So} &=  {s_{d-1}\cnorm n}e^{-R_n^2/2}\int_{0}^\infty (\rho^2+R_n^2)^{(d-2)/2}\rho e^{-\rho^2/2}d\rho \\
 &= s_{d-1}\cnorm n e^{-R_n^2/2}R_n^{d-2}\int_{0}^\infty \param{\param{{\rho}/{R_n}}^2+1}^{(d-2)/2}\rho e^{-\rho^2/2}d\rho.
\end{align*}
The integrand is bounded, and applying dominated convergence we have

\[
\limninf\param{{n e^{-R_n^2/2} R_n^{d-2}}}^{-1}\mean{\So}  = s_{d-1}\cnorm.
\]
Taking $R_n = \Roe \triangleq \sqrt{2 \log n + \param{{d-2}+\eps} \log \log n}$, we have
\[
e^{-R_n^2/2} = n^{-1}(\log n)^{-(d-2+\eps)/2}
\]
and so
\[
\limninf {n e^{-R_n^2/2} R_n^{d-2}} = 0
\]
which implies that
\[
    \mean{\So} \to 0.
\]

Finally, for every $0 \le k \le d-1$,
\[
\bk \le \So.
\]
Therefore,
\[
\limninf\mean{\bk} = 0,
\]
completing the proof.
\end{proof}


\bibliographystyle{plain}
\bibliography{refs}

\begin{thebibliography}{10}

\bibitem{adler_persistent_2010}
Robert~J. Adler, Omer Bobrowski, Matthew~S. Borman, Eliran Subag, and Shmuel
  Weinberger.
\newblock Persistent homology for random fields and complexes.
\newblock {\em Institute of Mathematical Statistics Collections}, 6:124--143,
  2010.

\bibitem{aronshtam_vanishing_2010}
Lior Aronshtam, Nathan Linial, Tomasz Luczak, and Roy Meshulam.
\newblock Vanishing of the top homology of a random complex.
\newblock {\em Arxiv preprint {arXiv:1010.1400}}, 2010.

\bibitem{babson_fundamental_2011}
Eric Babson, Christopher Hoffman, and Matthew Kahle.
\newblock The fundamental group of random 2-complexes.
\newblock {\em J. Amer. Math. Soc}, 24(1):1–28, 2011.

\bibitem{bobrowski_distance_2011}
Omer Bobrowski and Robert~J. Adler.
\newblock Distance functions, critical points, and topology for some random
  complexes.
\newblock {\em {arXiv:1107.4775}}, July 2011.

\bibitem{borsuk_imbedding_1948}
Karol Borsuk.
\newblock On the imbedding of systems of compacta in simplicial complexes.
\newblock {\em Fund. Math}, 35(217-234):5, 1948.

\bibitem{carlsson_topology_2009}
Gunnar Carlsson.
\newblock Topology and data.
\newblock {\em American Mathematical Society. Bulletin. New Series},
  46(2):255--308, 2009.

\bibitem{cohen_homotopical_2010}
Daniel~C. Cohen, Michael Farber, and Thomas Kappeler.
\newblock The homotopical dimension of random 2-complexes.
\newblock {\em Arxiv preprint {arXiv:1005.3383}}, 2010.

\bibitem{edelsbrunner_persistent_2008}
Herbert Edelsbrunner and John Harer.
\newblock Persistent homology - a survey.
\newblock In {\em Surveys on discrete and computational geometry}, volume 453
  of {\em Contemp. Math.}, pages 257--282. Amer. Math. Soc., Providence, {RI},
  2008.

\bibitem{edelsbrunner_book}
Herbert Edelsbrunner and John~L. Harer.
\newblock {\em Computational topology}.
\newblock American Mathematical Society, Providence, RI, 2010.
\newblock An introduction.

\bibitem{ghrist_barcodes:_2008}
Robert Ghrist.
\newblock Barcodes: the persistent topology of data.
\newblock {\em American Mathematical Society. Bulletin. New Series},
  45(1):61--75, 2008.

\bibitem{kahle_random_2011}
Matthew Kahle.
\newblock Random geometric complexes.
\newblock {\em Discrete \& Computational Geometry. An International Journal of
  Mathematics and Computer Science}, 45(3):553--573, 2011.

\bibitem{kahle_limit_2010}
Matthew Kahle and Elizabeth Meckes.
\newblock Limit theorems for {{B}etti} numbers of random simplicial complexes.
\newblock {\em 1009.4130}, September 2010.

\bibitem{meshulam_homological_2009}
Roy Meshulam and Nathan Wallach.
\newblock Homological connectivity of random k-dimensional complexes.
\newblock {\em Random Structures \& Algorithms}, 34(3):408–417, 2009.

\bibitem{niyogi_finding_2008}
Partha Niyogi, Stephen Smale, and Shmuel Weinberger.
\newblock Finding the homology of submanifolds with high confidence from random
  samples.
\newblock {\em Discrete \& Computational Geometry. An International Journal of
  Mathematics and Computer Science}, 39(1-3):419--441, 2008.

\bibitem{niyogi_topological_2011}
Partha Niyogi, Stephen Smale, and Shmuel Weinberger.
\newblock A topological view of unsupervised learning from noisy data.
\newblock {\em {SIAM} Journal on Computing}, 40(3):646, 2011.

\bibitem{pippenger_topological_2006}
Nicholas Pippenger and Kristin Schleich.
\newblock Topological characteristics of random triangulated surfaces.
\newblock {\em Random Structures \& Algorithms}, 28(3):247--288, May 2006.

\end{thebibliography}

\address{Robert J.\ Adler \\Electrical Engineering\\ Technion, Haifa,
Israel 32000 \\
\printead{e1}\\
\printead{u1}}

 \address{Omer Bobrowski \\Mathematics\\ Duke University,
120 Science Drive \\
Durham, NC 27708 \\
\printead{e2}\\
\printead{u2}}

\address{Shmuel Weinberger\\Mathematics \\ University of Chicago,
5734 S. University Ave\\
Chicago, IL 60637 \\
\printead{e5}\\
\printead{u5}}

\end{document}